\documentclass[12pt, reqno]{amsart}
\pdfoutput=1
\usepackage{amsrefs}
\usepackage{amssymb}
\usepackage{amsmath}
\usepackage{tikz-cd}
\usepackage[capitalise]{cleveref}
\usepackage{fullpage}

\newtheorem{theorem}{Theorem}[section]
\newtheorem{lemma}[theorem]{Lemma}
\newtheorem{proposition}[theorem]{Proposition}
\newtheorem{corollary}[theorem]{Corollary}
\theoremstyle{definition}
\newtheorem{definition}[theorem]{Definition}
\theoremstyle{remark}
\newtheorem{remark}[theorem]{Remark}

\newcommand{\C}{\mathcal C}
\newcommand{\T}{\mathcal T}
\newcommand{\V}{\mathcal V}
\newcommand{\NN}{\mathbb N}
\newcommand{\aA}{\mathfrak A}
\newcommand{\iso}{\cong}
\newcommand{\NO}{\mathcal{NO}}
\newcommand{\PO}{\mathcal{PO}}
\newcommand{\TNO}{\mathcal{TNO}}
\newcommand{\TPO}{\mathcal{TPO}}
\newcommand{\TNR}{\mathcal{TNR}}
\newcommand{\CR}{\mathcal{CR}}
\newcommand{\IR}{\mathcal{IR}}
\newcommand{\CCR}{\mathcal{CCR}}

\newcommand{\cat}{\sf}
\newcommand{\meet}{\wedge}
\newcommand{\Meet}{\bigwedge}
\newcommand{\join}{\vee}
\newcommand{\union}{\cup}
\newcommand{\inter}{\cap}
\newcommand{\Inter}{\bigcap}
\renewcommand{\:}{\colon}
\newcommand{\Not}{\neg}
\newcommand{\cl}{\mathrm{cl}}
\newcommand{\EV}{\quad\Longleftrightarrow\quad}

\allowdisplaybreaks

\begin{document}

\title{On the structure of modal and tense \\ operators on a boolean algebra}

\author{Guram Bezhanishvili}
\address{New Mexico State University}
\email{guram@nmsu.edu}

\author{Andre Kornell}
\address{Dalhousie University}
\email{akornell@dal.ca}

\thanks{Andre Kornell was supported by the Air Force Office of Scientific Research under Award No.~FA9550-21-1-0041.}


\begin{abstract}
    We initiate the study of the poset $\NO(B)$ of necessity operators on a boolean algebra $B$. We show that $\NO(B)$ is a meet-semilattice that need not be distributive. However, when $B$ is complete, $\NO(B)$ is necessarily a frame, which is spatial iff $B$ is atomic. In that case, $\NO(B)$ is a locally Stone frame. Dual results hold for the poset $\PO(B)$ of possibility operators. We also obtain similar results for the posets $\TNO(B)$ and $\TPO(B)$ of tense necessity and possibility operators on $B$. Our main tool is J\'onsson-Tarski duality, by which such operators correspond to continuous and interior relations on the Stone space of $B$.
\end{abstract}

\maketitle

\section{Introduction}

In the algebraic study of modal logic, necessity ($\Box$) is modeled by a function on a boolean algebra preserving finite meets, possibility ($\Diamond$) by a function preserving finite joins, and the two are related by $\Box=\lnot\Diamond\lnot$. In relational semantics, the necessity and possibility operators are modeled by a binary relation $R$ on a set $X$ of possible worlds. More precisely, for $x\in X$,   
\begin{eqnarray*}
    x\models\Box\varphi & \Longleftrightarrow & (\forall y)(x \mathrel{R} y \Longrightarrow y\models\varphi); \\
    x\models\Diamond\varphi & \Longleftrightarrow & (\exists y)(x \mathrel{R} y \mbox{ and } y\models\varphi).
\end{eqnarray*}

The algebraic and relational approaches are related to each other by the well-known {\em J\'onsson-Tarski duality}. 
We recall (see, e.g., \cite{CZ1997}*{p.~214}) that a {\em modal algebra} is a pair $(B,\Box)$ where $B$ is a boolean algebra and $\Box$ is a unary function on $B$ preserving finite meets or, equivalently, a pair $(B,\Diamond)$ where $\Diamond$ is a unary function preserving finite joins. A {\em modal homomorphism} between two modal algebras $(B_1,\Box_1)$ and $(B_2,\Box_2)$ is a boolean homomorphism $h\:B_1\to B_2$ such that $h(\Box_1 a)=\Box_2 h(a)$ or, equivalently, $h(\Diamond_1 a)=\Diamond_2 h(a)$ for each $a\in B_1$. Let $\cat{MA}$ be the category of modal algebras and modal homomorphisms.

To define the dual category of descriptive frames, we recall that a binary relation $R\subseteq X\times Y$ between Stone spaces $X$ and $Y$ is a {\em continuous relation} if $R[x]$ is closed for each $x\in X$ and $R^{-1}[U]$ is clopen for each clopen $U\subseteq Y$. A continuous function $X \to Y$ is then just a continuous relation that is also a function.

A {\em descriptive frame} is a pair $(X,R)$ where $X$ is a Stone space and $R$ is a continuous relation on $X$. 
Let $(X_1,R_1)$ and $(X_2,R_2)$ be two descriptive frames. A map $f\:X_1\to X_2$ is a {\em bounded morphism} or a {\em p-morphism} if $fR_1[x]=R_2[f(x)]$ for each $x\in X$. Let $\cat{DF}$ be the category of descriptive frames and continuous bounded morphisms.

\begin{theorem} [J\'onsson-Tarski duality]
    $\cat{MA}$ is dually equivalent to $\cat{DF}$.
\end{theorem}

The above duality originates in the well-known paper by J\'onsson and Tarski \cite{JT1951} and is a generalization of the celebrated Stone duality between boolean algebras and Stone spaces. We recall that if $(X,R)$ is a descriptive frame, then the necessity and possibility operators on the boolean algebra ${\rm Clop}(X)$ of clopen subsets of $X$ are defined by
\begin{eqnarray*}
    \Box_R U & = & \{ x \in X \mid R[x] \subseteq U \} = X \setminus R^{-1}[X \setminus U], \\
    \Diamond_R U & = & \{ x \in X \mid R[x] \cap U \ne \varnothing \} = R^{-1}[U]
\end{eqnarray*}
for each $U\in{\rm Clop}(X)$. These formulas interpret $\Box$ and $\Diamond$ in a descriptive frame.

\begin{definition}
    Let $B$ be a boolean algebra. A \emph{necessity operator} on $B$ is a unary function that preserves finite meets, and a \emph{possibility operator} on $B$ is a unary function that preserves finite joins.
    Let $\NO(B)$ be the set of these necessity operators and $\PO(B)$ the set of these possibility operators.
\end{definition}

We order both $\NO(B)$ and $\PO(B)$ by 
\[
f \le g \mbox{ iff } f(a) \le g(a) \mbox{ for all } a \in B.
\]
For this order, $\NO(B)$ and $\PO(B)$ are posets that are dual to each other, and it is straightforward to see that $\Box \mapsto \lnot \Box \lnot$ is the desired dual isomorphism.
Because of this, we will mainly concentrate on $\NO(B)$. Clearly, each result we obtain about $\NO(B)$ has a dual counterpart for $\PO(B)$, but we won't state them all explicitly. 

It is worth noting that $\NO(B)$ can be thought of as the set of endomorphisms of $B$ in the appropriate category. For this we recall another generalization of Stone duality developed by Halmos \cite{Hal1956}. A map $h\:B_1\to B_2$ between two boolean algebras is a {\em meet-hemimorphism} or a {\em join-hemimorphism} if it preserves finite meets or finite joins, respectively. In what follows, we will mainly work with meet-hemimorphisms and simply call them {\em hemimorphisms}. Let $\cat{BA^H}$ be the category of boolean algebras and hemimorphisms between them. Also, let $\cat{Stone^C}$ be the category of Stone spaces and continuous relations between them.

\begin{theorem} [Halmos duality]
    $\cat{BA^H}$ is dually equivalent to $\cat{Stone^C}$.
\end{theorem}

Clearly, $\NO(B)$ is the set of endomorphisms of $B$ in $\cat{BA^H}$. Dually, for a Stone space $X$, let $\CR(X)$ be the set of endomorphisms on $X$ in $\cat{Stone^C}$.
In other words, $\CR(X)$ is the set of continuous relations on $X$. Clearly $\CR(X)$ is a poset with respect to inclusion.

The next lemma is a consequence of J\'onsson-Tarski duality.
Recall that the Stone space $X$ of a boolean algebra $B$ is the space of ultrafilters of $B$ whose topology is generated by the basis $\{ \sigma(a) \mid a \in B\}$, where $\sigma(a) :=\{x \in X : a \in x\}$. Then $\sigma\:B\to{\rm Clop}(X)$ is the Stone isomorphism between $B$ and the boolean algebra ${\rm Clop}(X)$ of clopen subsets of $X$.

\begin{lemma} \label{lem: JT}
Let $B$ be a boolean algebra, and let $X$ be its Stone space. For $\Box\in\NO(B)$, define $R_\Box$ on $X$ by $x \ R_\Box \ y$ 
iff $\Box^{-1}[x]\subseteq y$ for all $x,y\in X$. The assignment $\Box \mapsto R_\Box$ is a dual isomorphism between the posets $\NO(B)$ and $\CR(X)$. 
\end{lemma}

\begin{proof}
    The assignment $\Box \mapsto R_\Box$ is a bijection $\NO(B) \to \CR(X)$ by J\'onsson-Tarski duality (see, e.g., \cite{CZ1997}*{Sec.~8.2}). Its inverse is $R \mapsto \Box_R$, where $\Box_R a =\sigma^{-1}(X\setminus R^{-1}[X\setminus\sigma(a)])$.
    We show that this bijection is a dual isomorphism.
   
    Let $\Box_1,\Box_2\in\NO(B)$, and assume $\Box_1\le \Box_2$. We prove that $R_{\Box_2} \subseteq R_{\Box_1}$. 
    Suppose $x,y\in X$ are such that $x \ R_{\Box_2} \ y$, and let $a \in \Box_1^{-1}[x]$. Since $\Box_1 a \in x$ and $\Box_1 a\le \Box_2 a$, we have $\Box_2 a\in x$, so $a\in y$ as $x \ R_{\Box_2} \  y$. This proves that $x \ R_{\Box_1} \ y$, and so $R_{\Box_2} \subseteq R_{\Box_1}$. Next, let $R,S\in\CR(X)$, and assume $R\subseteq S$. We prove that $\Box_S\le \Box_R$. Suppose $a\in B$. From $R \subseteq S$ it follows that $R^{-1}[X \setminus \sigma(a)] \subseteq S^{-1}[X\setminus \sigma(a)]$. Therefore, $X\setminus S^{-1}[X\setminus \sigma(a)] \subseteq X\setminus R^{-1}[X\setminus \sigma(a)]$, and so $\Box_S a \le \Box_R a$ because $\sigma$ is an isomorphism. Thus, $\Box_S \le \Box_R$.
\end{proof}

Consequently, the poset $\PO(B)$ is isomorphic to the poset $\CR(X)$. We study $\NO(B)$ and hence $\PO(B)$
by studying $\CR(X)$. 
For an arbitrary Stone space $X$, we show that $\CR(X)$ is a join-semilattice that is not necessarily distributive. Thus, $\NO(B)$ is a meet-semilattice that is not necessarily distributive.
On the other hand, when $X$ is extremally disconnected, we show that $\CR(X)$ is actually a coframe, and hence for a complete boolean algebra $B$, we obtain that $\NO(B)$ is a frame, albeit not a spatial one. In fact, it is spatial iff $B$ is furthermore atomic, in which case $\NO(B)$ is even a locally Stone frame. It is a boolean algebra iff $B$ is finite.

In the last section of the paper, we also study the poset of tense operators on a boolean algebra $B$. Tense operators are used in modal logic to model future and past modalities. A \emph{tense possibility operator} is a possibility operator that has a conjugate in the sense of J\'onsson and Tarski \cite{JT1951}, and a \emph{tense necessity operator} is simply the dual of a tense possibility operator. 

Tense operators on $B$ dually correspond to those continuous relations $R$ on the Stone space $X$ of $B$ that additionally satisfy the condition that $R[U]$ is clopen for each clopen $U\subseteq X$. Following \cite{BGHJ2019}*{Def.~5.1}, we refer to such relations as {\em interior relations}. This name is motivated by the fact that if $R$ is a function, then it is an interior relation iff it is an interior function (that is, a continuous and open function \cite{RS1963}*{p.~99}). If $R$ is an interior relation on $X$, then the tense operators $\Box_F$~(future necessity), $\Diamond_F$~(future possibility), $\Box_P$~(past necessity), and $\Diamond_P$~(past possibility) are defined on ${\rm Clop}(X)$ as follows:
\begin{eqnarray*}
    \Diamond_F U &=& R^{-1}[U], \quad \Box_F U = X \setminus R^{-1}[X \setminus U], \\
    \Diamond_P U &=& R[U], \quad \quad \Box_P U = X \setminus R[X \setminus U]
\end{eqnarray*}
for each $U\in{\rm Clop}(X)$.

Let $\TNO(B)$ be the poset of tense necessity operators on $B$. Similarly to \cref{lem: JT}, $\TNO(B)$ is dual to the poset $\IR(X)$ of interior relations on the Stone space $X$ of $B$. Using this correspondence, we will show that the results we obtained for $\NO(B)$ have obvious analogues for $\TNO(B)$. In particular, we will prove that in general $\TNO(B)$ is only a meet-semilattice that is not necessarily distributive. However, when $B$ is a complete boolean algebra, then $\TNO(B)$ is a frame, which is spatial iff $B$ is furthermore atomic. In such a case, $\TNO(B)$ is itself a complete and atomic boolean algebra. Of course, dual results hold for the poset $\TPO(B)$ of tense possibility operators on $B$.

\section{Basic structure of $\NO(B)$}

In this section, we prove that for each boolean algebra $B$, the poset $\NO(B)$ is a bounded meet-semilattice but that, in general, $\NO(B)$ may be neither distributive nor a lattice. Proving that $\NO(B)$ is a bounded meet-semilattice is straightforward. That $\NO(B)$ may be neither distributive nor a lattice is proved by first showing that there is a Stone space $X$ such that the join-semilattice $\CR(X)$ is neither distributive nor a lattice, and then invoking \cref{lem: JT}.

\begin{proposition}\label{prop: semilattice}
For each boolean algebra $B$, the poset $\NO(B)$ is a bounded meet-semilat\-tice.
\end{proposition}

\begin{proof}
The top element of $\NO(B)$ maps everything to $1$, and the bottom element 
maps $1$ to $1$ and everything else to $0$. For $f, g \in \NO(B)$, the meet $f \meet g$ is given componentwise, i.e., $(f \meet g)(a) = f(a) \meet g(a)$ for all $a \in B$.
\end{proof}

The above proposition together with \cref{lem: JT} yield that $\CR(X)$ is a bounded join-semilattice for each Stone space $X$. We can also observe this directly: the bottom element of $\CR(X)$ is the empty relation, the top element 
is all of $X \times X$, and the join of $R,S \in \CR(X)$ is $R \union S$.
We next show that $\CR(X)$ may in general be neither distributive nor a lattice. For this we require a series of lemmas.

For each space $X$, let $D_X = \{ (x,x) \mid x \in X \}$ be the diagonal relation on $X$.

\begin{lemma} \label{lem: CR under D}
Let $X$ be a Stone space. The continuous relations underneath the diagonal $D_X \subseteq X \times X$ are exactly the relations of the form $D_X \inter (U \times U)$, where $U \subseteq X$ is clopen.
\end{lemma}

\begin{proof}
First, suppose that $R\subseteq D_X$ is continuous. Then $U:=R^{-1}[X]$ is clopen, and $R=D_X \inter (U \times U)$. Conversely, suppose that $R=D_X \inter (U \times U)$ for some clopen $U \subseteq X$. Clearly $R[x]$ is closed for each $x\in X$. Moreover, if $V$ is clopen, then $R^{-1}[V]=U\cap V$, so $R^{-1}[V]$ is clopen. Thus, $R$ is continuous.
\end{proof}

Let $\alpha \NN$ be the one-point compactification of $\NN$.

\begin{definition}\label{def: DRS}
Define $D, R, S \subseteq \alpha \NN \times \alpha \NN$ as follows:
\begin{enumerate}
\item $D = D_{\alpha\NN} $;
\item $R = \{ (2n, 2n) \mid n \in \NN  \} \union \{(2n+1, 2n + 3) \mid n \in \NN\} \union \{(\infty, \infty)\}$;
\item $S = \{ (2n, 2n+2) \mid n \in \NN\} \union \{(2n+1, 2n+1) \mid n \in \NN\} \union \{(\infty, \infty)\}$.
\end{enumerate}
\end{definition}

\begin{lemma}\label{lem: DRS interior}
The relations $D$, $R$, and $S$ are functions that are both continuous and open. In particular, they are continuous relations.
\end{lemma}

\begin{proof}
All three of these relations are injective functions $\alpha \NN \to \alpha \NN$. We show that these functions are continuous. For each finite subset $U \subseteq \NN$, the sets $D^{-1}[U]$, $R^{-1}[U]$, and $S^{-1}[U]$ are also finite subsets of $\NN$ since $D$, $R$, and $S$ are injections that map $\infty$ to $\infty$. Because the finite subsets of $\NN$ and their complements in $\alpha \NN$ form a basis for $\alpha \NN$, we conclude that $D$, $R$, and $S$ are continuous functions and hence continuous relations. Furthermore, they are open functions because they are continuous injections on a compact Hausdorff space and they have open ranges.
\end{proof}

\begin{proposition}
We have the following:
\begin{enumerate}
\item $R \cap S$ is not continuous;
\item $R$ and $S$ have a meet in $\CR(\alpha \NN)$.
\end{enumerate}
\end{proposition}

\begin{proof}
(1) We have $R\cap S=\{(\infty,\infty)\}$. Since $\{\infty\}$ is not clopen in $\alpha\NN$, we conclude from \cref{lem: CR under D} that $R\cap S$ is not continuous. 

(2) It follows from (1) and \cref{lem: CR under D} that the only continuous relation under $R\cap S$ is $\varnothing$. Thus, $\varnothing$ is the meet of $R$ and $S$ in $\CR(\alpha\NN)$.
\end{proof}

We recall (see, e.g., \cite[p.~99]{Gra1978}) that a join-semilattice $L$ is {\em distributive} if for each $a,b,c\in L$, the inequality $a\le b\vee c$ implies that there are $b'\le b$ and $c'\le c$ such that $a=b'\vee c'$. This notion of distributivity dualizes to meet-semilattices, and it reduces to the familiar notion in the presence of both meets and joins.

\begin{theorem}\label{thm: counterexamples}
We have the following:
\begin{enumerate}
\item $D$ and $R$ do not have a meet in $\CR(\alpha \NN)$;
\item $D$ and $S$ do not have a meet in $\CR(\alpha \NN)$;
\item there do not exist continuous relations $R' \subseteq R$ and $S' \subseteq S$ such that $R' \union S' = D$.
\end{enumerate}
Thus, the bounded join-semilattice $\CR(\alpha \NN)$ is neither distributive nor a lattice.
\end{theorem}

\begin{proof}
(1) We have $D\cap R = \{ (2n, 2n) \mid n \in \NN  \} \union \{(\infty, \infty)\}$. Since the only clopen subsets of $\{ 2n \mid n \in \NN \} \union \{\infty\}$ are the finite subsets of $\{ 2n \mid n \in \NN \}$, it follows from \cref{lem: CR under D} that $D\cap R$ is not a continuous relation and that the only continuous relations under $D\cap R$ are finite subsets of $\{ (2n, 2n) \mid n \in \NN  \}$. Thus, there is no greatest continuous relation under $D\cap R$, which implies that the meet of $D$ and $R$ does not exist in $\CR(\alpha \NN)$.

(2) This is proved similarly to (1). 

(3) Let $R',S'\in\CR(X)$ be such that $R'\subseteq R$,  $S'\subseteq S$, and $R'\cup S'=D$. Then $R',S'\subseteq D$. Therefore, $R'\subseteq R\cap D$ and $S'\subseteq S\cap D$. As in the proof of (1), this forces $R'$ and $S'$ to be finite subsets of $D$, contradicting $R'\cup S'=D$.

That $\CR(\alpha \NN)$ is not a lattice follows from (1), and that it is not a distributive join-semilattice follows from (3).
\end{proof}

\begin{remark}
We see that in general $\CR(X)$ need not be a lattice. However, if $X$ is such that $\CR(X)$ is a lattice, then $\CR(X)$ must be distributive by the proof of \cref{thm: coframe}, which carries over in its entirety.
\end{remark}

\begin{corollary}
Let $\aA$ be the boolean algebra of finite and cofinite subsets of $\NN$. The bounded meet-semilattice $\NO(\aA)$ is neither distributive nor a lattice.
\end{corollary}

\begin{proof}
This follows immediately from \cref{thm: counterexamples,lem: JT} because the Stone space of $\aA$ is well-known to be homeomorphic to $\alpha \NN$.
\end{proof}

\section{The structure of $\NO(B)$ for $B$ complete}

In this section, we show that if $B$ is a complete boolean algebra, then the poset $\NO(B)$ has a much richer structure. Specifically, we prove that $\NO(B)$ is a complete lattice that satisfies the join-infinite distributive law and hence is a frame (see, e.g., \cite[p.~10]{PP2012} and below). While it is straightforward to see that $\NO(B)$ is a complete lattice, to prove that $\NO(B)$ is a frame, we again first show that $\CR(X)$ is a coframe (see, e.g., \cite{PP2012}*{p.~10} and below), where $X$ is the Stone dual of $B$, and then invoke \cref{lem: JT}.
In addition, we prove that $\NO(B)$ is a spatial frame iff $B$ is atomic, in which case $\NO(B)$ is a locally Stone frame. This result is proved by utilizing the well-known Vietoris construction in topology (see, e.g., \cite{Joh1982}*{Sec.~III.4}). 

We begin by recalling (see, e.g., \cite[pp.~90--92]{Hal1963}) that by Stone duality, complete boolean algebras correspond to extremally disconnected Stone spaces, where a space $X$ is {\em extremally disconnected} if the closure of each open set is open.

\begin{proposition} \label{prop: NO of CBA}
For each complete boolean algebra $B$, the poset $\NO(B)$ is a complete lattice.
\end{proposition}

\begin{proof}
Let $\mathcal S \subseteq \NO(B)$. Since $B$ is complete, we can define the meet $\bigwedge \mathcal S$ componentwise: $(\bigwedge \mathcal S)(a) = \bigwedge \{ f(a) \mid f\in\mathcal S\}$ for all $a \in B$. Thus, $\NO(B)$ is a complete lattice.
\end{proof}

By \cref{lem: JT}, it follows that $\CR(X)$ is a complete lattice for each extremally disconnected Stone space $X$. We next give an explicit description of joins in $\CR(X)$. For this we prove the following lemma, which we will use later as well.

\begin{lemma} \label{lem: auxiliary}
Let $X$ be a Stone space, $P \subseteq X \times X$, and $R=\cl P$. Then $R^{-1}[V]=\cl (P^{-1}[V])$ for each clopen $V \subseteq X$.    
\end{lemma}

\begin{proof}
Clearly $P^{-1}[V] \subseteq R^{-1}[V]$, and hence $\cl P^{-1}[V] \subseteq R^{-1}[V]$. For the converse inclusion, suppose $x\notin \cl P^{-1}[V]$. Then there is a clopen set $U$ such that $x\in U$ and $U\cap P^{-1}[V] = \varnothing$. Thus, $P\cap (U\times V) = \varnothing$, so $\cl P \cap (U\times V) = \varnothing$, or equivalently, $R \cap (U\times V) = \varnothing$. Therefore, $U \cap R^{-1}[V] = \varnothing$, and so $x\notin R^{-1}[V]$.    
\end{proof}

\begin{lemma}\label{lem:joins}
Let $X$ be an extremally disconnect Stone space. For all $\mathcal{S} \subseteq \CR(X)$, we have that
$\bigvee \mathcal{S} = \cl(\bigcup \mathcal{S})$.
\end{lemma}

\begin{proof}
Let $R = \cl(\bigcup \mathcal{S})$. It is enough to show that $R$ is continuous. It is clear that $R$ is closed, so $R[x]$ is closed for each $x \in X$. Let $U \subseteq X$ be clopen. By \cref{lem: auxiliary},
\[
R^{-1}[U] = \cl\left(\bigcup \left\{ S^{-1}[U] \mid S \in \mathcal{S} \right\}\right).
\]
Since $X$ is extremally disconnected and $\bigcup \left\{ S^{-1}[U] \mid S \in \mathcal{S} \right\}$ is open, we conclude that $R^{-1}[U]$ is clopen. Therefore, $R$ is a continuous relation.
\end{proof}

Although meets exist in $\CR(X)$ even for infinite families of relations, the meet of two continuous relations need not be equal to their intersection. To see this, let $\beta\mathbb N$ be the Stone-\v Cech compactification of $\mathbb N$, and let $\NN^* = \beta\NN \setminus \NN$ be the remainder.

\begin{proposition} \label{prop: R intersection S}
There exist $R, S \in \CR(\beta \NN)$ such that $ R \wedge S \neq R \cap S$.
\end{proposition}

\begin{proof}
Let $R = \{(n_1, n_2) \mid n_1 \leq n_2 \in \NN\} \cup (\beta \NN \times \NN^*)$. For each $n \in \NN$, we have that $R[n] = \{n, n+1, \ldots\} \cup \NN^*$ and that $R^{-1}[n] = \{0, \ldots, n\}$; both sets are closed. For each $x \in \NN^*$, we have that $R[x] = \NN^*$ and that $R^{-1}[x] = \beta \NN$; both sets are closed. Thus, under both $R$ and it converse $R^\dagger$, the image of a point is closed.

Each clopen subset of $\beta \NN$ is the closure of some subset $A \subseteq \NN$. For each nonempty $A \subseteq \NN$, we have that $$R[\cl(A)] =  \{\min A, \min A + 1, \ldots\} \cup \NN^* = \beta \NN \setminus \{0, \ldots, \min A -1 \},$$ and this set is clopen. For each nonempty finite $A \subseteq \NN$, we have that $R^{-1}[\cl(A)] = R^{-1}[A] = \{0, \ldots, \max A\}$, and for each infinite $A \subseteq \NN$, we have that $R^{-1}[\cl(A)] = \beta \NN$; both sets are clopen. Thus, under both $R$ and $R^\dagger$, the inverse image of a clopen set is clopen.

We have shown that $R$ and $R^\dagger$ are continuous relations. However, the relation $R \cap R^\dagger = \{(n, n) \mid n \in \NN\} \cup (\NN^* \times \NN^*)$ is not continuous. Indeed, if $A$ is the set of even natural numbers, then $\cl(A)$ is clopen, but $(R \cap R^\dagger)^{-1}[\cl(A)] = A \union \NN^*$ is not clopen because its complement is the set of odd natural numbers. Since $R \cap R^\dagger$ is not continuous, it must be distinct from $R  \wedge R^\dagger$.
\end{proof}

We now turn to showing that $\CR(X)$ is a coframe for every extremally disconnected Stone space $X$. Recall  
that a complete lattice $L$ is a \emph{frame} if it satisfies the join-infinite distributive law $ a \wedge \bigvee S = \bigvee\{ a \wedge b \mid b \in S\} $ for all $S \subseteq L$. Dually, $L$ is a {\em coframe} if it satisfies the meet-infinite distributive law $ a \vee \bigwedge S = \bigwedge\{ a \vee b \mid b \in S\} $ for all $S \subseteq L$.

\begin{theorem} \label{thm: coframe}
If $X$ is an extremally disconnected Stone space, then $\CR(X)$ is a coframe.
\end{theorem}

\begin{proof}
Recalling that finite joins in $\CR(X)$ are unions, we must show that $R \union \Meet_\alpha S_\alpha = \Meet_\alpha (R \union S_\alpha)$ for all $R \in \CR(X)$ and $\{S_\alpha\}_{\alpha \in I} \subseteq \CR(X)$.
Trivially, $R \union \Meet_\alpha S_\alpha \subseteq \Meet_\alpha (R \union S_\alpha)$. Assume that $R \union \Meet_\alpha S_\alpha \subsetneq \Meet_\alpha (R \union S_\alpha)$, and let $(x, y) \in \big(\Meet_\alpha (R \union S_\alpha)\big) \setminus \big(R \union \Meet_\alpha S_\alpha\big)$. The continuous relation $R \union \Meet_\alpha S_\alpha$ is a closed subset of $X \times X$, and thus the point $(x, y)$ has a basic neighborhood $U \times V$ that is disjoint from $R \union \Meet_\alpha S_\alpha$, where $U$ and $V$ are clopen subsets of $X$. Thus, the sets $D_X \inter (U \times U)$ and $D_X \inter (V \times V)$ are continuous relations on $X$ by \cref{lem: CR under D}. Therefore, the composition $T = \big(D_X \inter (V \times V)\big) \circ \big(\Meet_\alpha (R \union S_\alpha)\big) \circ \big(D_X \inter (U \times U)\big)$ is a continuous relation on $X$.

We show that that $T = \big(\Meet_\alpha (R \union S_\alpha)\big) \cap (U \times V)$. For all $(x', y') \in X \times X$,
\begin{align*}&
(x', y') \in \big(D_X \inter (V \times V)\big) \circ \left(\Meet_\alpha (R \union S_\alpha)\right) \circ \big(D_X \inter (U \times U)\big)
\\ & \EV (x', x'') \in D_X \inter (U \times U) \; \text{and} \; (x'', y'') \in \Meet_\alpha (R \union S_\alpha) \; \text{and} \; (y'', y') \in D_X \inter (V \times V)
\\ & \hspace {70ex} \text{for some}\; x'', y'' \in X
\\ & \EV x'\in U \; \text{and} \; (x', y') \in \Meet_\alpha (R \union S_\alpha) \; \text{and} \;  y' \in V
\\ &\EV
(x', y') \in \left(\Meet_\alpha (R \union S_\alpha)\right) \cap (U \times V).
\end{align*}
Thus, $T$ is disjoint from $R \union \Meet_\alpha S_\alpha$, and in particular, $T$ is disjoint from $R$. However, $T \subseteq \Meet_\alpha (R \union S_\alpha) \subseteq \Inter_\alpha (R \union S_\alpha)$; hence $T \subseteq S_\alpha$ for each $\alpha \in I$. Therefore, $T \subseteq \Meet_\alpha S_\alpha$.

We obtain a contradiction: $(x,y) \in \Meet_\alpha (R \union S_\alpha)$ and $(x,y) \in U \times V$, so 
\[
(x, y) \in \left(\Meet_\alpha (R \union S_\alpha)\right) \cap (U \times V) = T \subseteq \Meet_\alpha S_\alpha \subseteq R \union \Meet_\alpha S_\alpha.
\]
Therefore, $R \union \Meet_\alpha S_\alpha = \Meet_\alpha (R \union S_\alpha)$. We conclude that $\CR(X)$ is a coframe, as claimed.
\end{proof}

\begin{definition}
    Let $X$ be a Stone space. We call a relation $R$ on $X$ {\em cocontinuous} provided $R$ is the set-theoretic complement of a continuous relation on $X$. Let $\CCR(X)$ be the poset of cocontinuous relations on $X$, which are ordered by inclusion.
\end{definition}

As an immediate consequence of \cref{thm: coframe}, we obtain that $\CCR(X)$ is a frame. Moreover, for $R,S\in\CCR(X)$, we have that $R \wedge S= R\cap S$, and for $\{R_\alpha\}\subseteq\CCR(X)$, we have that $\bigvee_\alpha R_\alpha$ is the least cocontinuous relation containing $\bigcup_\alpha R_\alpha$. In particular, if $\bigcup_\alpha R_\alpha \in \CCR(X)$, then $\bigvee_\alpha R_\alpha = \bigcup_\alpha R_\alpha$. 

We recall (see, e.g., \cite{Joh1982}*{p.~101}) that a frame $L$ is {\em zero-dimensional} if the complemented elements of $L$ join-generate $L$.

\begin{proposition} \label{prop: zero-dim}
$\CCR(X)$ is a zero-dimensional frame.
\end{proposition}

\begin{proof}
    Since each clopen of $X\times X$ is a continuous relation on $X$, the clopens of $X\times X$ are complemented elements of $\CCR(X)$. Because $\{ U\times V \mid U,V$ clopen in $X\}$ is a basis for the topology on $X\times X$, each open subset $W$ of $X\times X$ is a union of the sets $U\times V$ contained in $W$. If $W\in\CCR(X)$, then this union is the join in $\CCR(X)$. Thus, $\CCR(X)$ is zero-dimensional. 
\end{proof}

We recall that a frame $L$ is {\em spatial} if $L$ is isomorphic to the frame of open sets of a topological space. It is well known (see, e.g., \cite[p.~18]{PP2012}) that $L$ is spatial iff each element of $L$ is a meet of meet-prime elements of $L$, where $m\in L\setminus\{1\}$ is {\em meet-prime} if $a\wedge b\le m$ implies $a\le m$ or $b\le m$. We next show that in general $\CCR(X)$ need not be spatial. For this we use that $R\in\CCR(X)$ is meet-prime iff its complement $R^c$ is join-prime in $\CR(X)$, where join-prime elements are defined dually.

\begin{lemma} \label{lem: join-prime}
Let $X$ be an extremally disconnected Stone space. Then $R\in\CR(X)$ is join-prime iff $R=\{(x,y)\}$, where $x$ is an isolated point of $X$ and $y$ is any point of $X$.
\end{lemma}

\begin{proof}
    Let $x$ be isolated. Then it is clear that $\{(x,y)\}\in\CR(X)$ for each $y\in X$, and it is obvious that $\{(x,y)\}$ is join-prime in $\CR(X)$.
    Conversely, let $R$ be join-prime in $\CR(X)$. 
    We first show that $R$ must be a singleton. Suppose $R$ contains at least two distinct points $(x,y)$ and $(x',y')$. Since $(x,y)\ne (x',y')$ and $X\times X$ is a Stone space, there is a clopen relation $S$ such that $(x,y)\in S$ and $(x',y')\notin S$. Then $S,S^c\in\CR(X)$ and $R\subseteq S\cup S^c$, but $R\not\subseteq S,S^c$. Therefore, $R$ is not join-prime. The obtained contradiction proves that $R$ is a singleton, so let $R=\{(x,y)\}$. Because $R$ is continuous, $R^{-1}[X]$ is clopen. But $R^{-1}[X]=\{x\}$. Thus, $x$ must be an isolated point of $X$.
\end{proof}

\begin{theorem}\label{thm: CCR is spatial}
Let $X$ be an extremally disconnected Stone space. The following are equivalent:
\begin{enumerate}
    \item $\CCR(X)$ is a spatial frame;
    \item the isolated points of $X$ are dense in $X$;
    \item $X$ is homeomorphic to the Stone-\v Cech compactification $\beta M$ for some set $M$.
\end{enumerate}
\end{theorem}

\begin{proof}
    (3)$\Rightarrow$(2) is obvious and (2)$\Rightarrow$(3) is well known (see, e.g., \cite[p.~96]{GJ1960}). To see (2)$\Rightarrow$(1), we must show that 
    $\CCR(X)$ is meet-generated by meet-primes, which is equivalent to $\CR(X)$ being join-generated by join-primes. Let $R, S\in\CR(X)$ be such that $R \not\subseteq S$. Then there are clopen $U,V \subseteq X$ such that $(U \times V) \cap R \neq \varnothing$ and $(U \times V) \cap S = \varnothing$. Thus, $U \cap R^{-1}[V] \ne \varnothing$ and $U \cap S^{-1}[V] = \varnothing$. 
    Since $U \cap R^{-1}[V]$ is a nonempty clopen, it follows from (2) that there is an isolated point $x \in U \cap R^{-1}[V]$. But then there is $y\in V$ such that $(x,y) \in R$. However, $(x,y) \notin S$. Therefore, for all $R, S \in \CR(X)$, if $R \not\subseteq S$ then there exists a singleton $\{(x,y)\} \subseteq R \setminus S$ such that $x$ is isolated. Since $\{(x,y)\}$ is join-prime in $\CR(X)$ by \cref{lem: join-prime}, we conclude that $\CR(X)$ is join-generated by join-primes, and hence $\CCR(X)$ is spatial.
   
    To see (1)$\Rightarrow$(2), let $U$ be a nonempty clopen subset of $X$. Then $U\times U$ is a nonempty clopen relation on $X$, so $U\times U\in\CR(X)$. Since $\CCR(X)$ is spatial, $\CR(X)$ is join-generated by join-irreducibles. Therefore, by \cref{lem: join-prime}, there is an isolated point $x$ of $X$ such that $(x,y)\in U\times U$ for some $y\in X$. Thus, $x\in U$, and hence the isolated points are dense in~$X$. 
\end{proof}

In particular, if $X$ has no isolated points, (e.g., if $X$ is the Gleason cover of $[0,1]$), then $\CCR(X)$ is not spatial.

\begin{corollary}
    Let $B$ be a complete boolean algebra. Then $\NO(B)$ is a zero-dimensional frame, which is spatial iff $B$ is atomic.
\end{corollary}

\begin{proof}
    Let $X$ be the Stone space of $B$. Then $\NO(B)$ is dual to $\CR(X)$, which is dual to $\CCR(X)$, and hence $\NO(B)$ is isomorphic to $\CCR(X)$. Thus, $\NO(B)$ is zero-dimensional by \cref{prop: zero-dim}, and it is spatial iff $B$ is atomic by \cref{thm: CCR is spatial}.
\end{proof}

We next show that if $X = \beta M$ for some set $M$, then $\CCR(X)$ is not only a zero-dimensional spatial frame but is in fact the frame of opens of a locally Stone space. 
A space is {\em locally Stone} if it is locally compact, Hausdorff, and zero-dimensional in the sense that it has a basis of clopen sets. In other words, locally Stone spaces are obtained by changing ``compact" to ``locally compact" in the definition of a Stone space. 

We recall that a frame $L$ is {\em algebraic} if each element of $L$ is a join of compact elements. We call $L$ a {\em locally Stone frame} if $L$ is zero-dimensional and algebraic. The next theorem belongs to folklore. We give a short proof for the reader's convenience.

\begin{theorem} \label{thm: locally Stone}
A frame is locally Stone iff it is isomorphic to the frame of opens of a locally Stone space.    
\end{theorem}

\begin{proof}
    Let $L$ be a frame.
    First suppose that $L$ is isomorphic to the frame $\Omega(X)$ of open sets of some locally Stone space $X$. Since $X$ is zero-dimensional, so is $\Omega(X)$. Let $U\in\Omega(X)$. Since $X$ is locally compact, for each $x\in U$, there are an open $V$ and a compact $K$ such that $x\in V\subseteq K\subseteq U$. Because $X$ is zero-dimensional, there is a clopen set $G$ such that $x\in G\subseteq V$. Then $G\subseteq K$, and hence $G$ is compact. Thus, $U$ is a union of compact clopens, yielding that $\Omega(X)$ is algebraic.

    Conversely, suppose that $L$ is a locally Stone frame. Then $L$ is spatial because it is algebraic (see, e.g., \cite[Rem.~3.2]{MZ2003}). Thus,  $L\cong\Omega(X)$, where $X$ is the space of points (completely prime filters) of $L$ and the isomorphism $\zeta\:L\to\Omega(X)$ is defined by $\zeta(a)=\{x\in X \mid a\in x\}$. But then $a\in L$ is complemented iff $\zeta(a)$ is clopen. Therefore, $L$ being zero-dimensional implies that $X$ is zero-dimensional.
    
    We show that $X$ is locally compact. Let $U\in\Omega(X)$. Then $U=\zeta(a)$ for some $a\in L$. Since $L$ is algebraic, $a$ is a join of compact elements. Therefore, for each $x\in U$ there is a compact $b\in L$ such that $x\in\zeta(b)\subseteq U$. But $L$ is also zero-dimensional, so $b$ is a join of complemented elements, each of which must be compact because $b$ is compact. Thus, there is a compact complemented $c\in L$ such that $x\in\zeta(c)\subseteq U$. Since $\zeta(c)$ is compact and clopen, $X$ is locally compact. 
    
    It is left to show that $X$ is Hausdorff, but this is obvious since $X$ is $T_0$ and hence Hausdorff because it is zero-dimensional. Thus, every locally Stone frame is isomorphic to the frame of opens of a locally Stone space, and the theorem is proved.
\end{proof}

\begin{remark}
        We recall (see, e.g., \cite[p.~54]{GHKLMS2003}) that a frame $L$ is {\em continuous} if each element is the join of elements way below it. Since for compact $a\in L$, $a\ll b$ iff $a\le b$, every algebraic frame is continuous. If $L$ is zero-dimensional, the converse is also true. Indeed, let $a\ll b$. Then there is $c$ with $a\ll c\ll b$ (see, e.g., \cite[p.~56]{GHKLMS2003}). Since $L$ is zero-dimensional, $c$ is the join of complemented elements below it. Because this set is directed and $a \ll c$, there is a complemented $d$ such that $a\le d\ll b$. Since $d$ is complemented and $d\ll b$, it is compact. Therefore, $b$ is a join of compact elements, and so $L$ is algebraic. We thus obtain that a frame $L$ is locally Stone iff $L$ is continuous and zero-dimensional. Since $L$ is continuous iff its space of points is locally compact (see, e.g., \cite[p.~423]{GHKLMS2003}), defining a locally Stone frame to be a continuous zero-dimensional frame is more in line with the definition of a locally Stone space given above.
\end{remark}

\begin{remark}
        Let $\sf Stone$ be the category of Stone spaces and continuous maps, and let $\sf StoneFrm$ be the category of Stone frames and frame homomorphisms. Then $\sf Stone$ is dually equivalent to $\sf StoneFrm$, and this dual equivalence can be thought of as the frame-theoretic version of the celebrated Stone duality. This dual equivalence readily generalizes to the setting of locally Stone spaces and locally Stone frames. We omit the details because this duality is not essential for the paper.
\end{remark}

In order to prove that $\CCR(\beta M)$ is a locally Stone frame, 
we recall the construction of the Vietoris space (see, e.g., \cite[pp.~111--112]{Joh1982}).
Let $X$ be compact and Hausdorff. The {\em Vietoris space} of $X$ is the set $\mathcal V(X)$ of closed subsets of $X$ whose topology is generated by the subbasis $\{ \Box_U \mid U \in \Omega(X) \} \cup \{ \Diamond_U \mid U \in \Omega(X) \}$, where 
$
\Box_U = \{ F \in \mathcal V(X) \mid F \subseteq U \}
$
and
$
\Diamond_U = \{ F \in \mathcal V(X) \mid F \cap U \ne \varnothing \}.
$
This topology makes $\mathcal V(X)$ a compact Hausdorff space (see, e.g., \cite[p.~244]{Eng1989}).

Moreover, if $X$ is a Stone space, then so is $\mathcal V(X)$ (see, e.g., \cite[p.~380]{Eng1989}). In this case, we can restrict our attention to $\Box_U$ and $\Diamond_U$ where $U$ ranges over clopen subsets of $X$. In addition, for clopen $U$, we have $\Box_{U^c} = (\Diamond_U)^c$ and $\Diamond_{U^c} = (\Box_U)^c$. Furthermore, $\subseteq$ is a closed relation on $\V(X)$ (see, e.g., \cite[p.~57]{Hoff1985}).

Let $\cat{KHaus}$ be the category of compact Hausdorff spaces and continuous functions, and let $\cat{KHaus^C}$ be the category of compact Hausdorff spaces and continuous relations. The Vietoris construction extends to a functor $\cat{KHaus^C} \to \cat{KHaus}$ that is right adjoint to the inclusion functor $\cat{KHaus} \hookrightarrow \cat{KHaus^C}$. The counit of this adjunction is the continuous relation $\ni\:\V(X) \to X$. Thus, for each continuous relation $R\: X \to Y$, there is a unique continuous function $\rho_R\: X \to \V(Y)$ such that ${\ni} \circ \rho_R = R$:
\[
\begin{tikzcd}
X
\arrow{rd}{R}
\arrow[dotted]{d}[swap]{\rho_R}
&
\\
\V(Y)
\arrow{r}[swap]{\ni}
&
Y
\end{tikzcd}
\]
The function $\rho_R$ is defined by $\rho_R(x) = R[x]$ for $x \in X$. These functors restrict to an adjoint pair of functors between the categories $\cat{Stone}$ and $\cat{Stone^C}$. 

As a corollary of \cref{prop: R intersection S}, we find the following unexpected property of the Vietoris construction.

\begin{proposition}
The function $\cap \colon \V(\beta \NN) \times \V(\beta \NN) \to \V(\beta \NN)$ is not continuous.
\end{proposition}

\begin{proof}
Assume that $\cap \colon \V(\beta \NN) \times \V(\beta \NN) \to \V(\beta \NN)$ is continuous, and let $R, S \in \CR(\beta \NN)$ be such that $R \wedge S \neq R \cap S$, as in \cref{prop: R intersection S}. Let $\rho_R, \rho_S\: \beta \NN \to \V(\beta \NN)$ be the continuous functions defined by $\rho_R (x) = R[x]$ and $\rho_S (x) = S[x]$ for $x \in \beta \NN$. Let $f\: \beta \NN \to \V(\beta \NN)$ be the function defined by $f(x) = \rho_R(x) \cap \rho_S(x)$ for $x \in \beta \NN$. Then $f$ is continuous by assumption. Finally, let $T$ be the continuous relation defined by $T = {\ni} \circ f$. In other words, $T[x] = f(x)$ for all $x \in \beta \NN$.

We now calculate that $T[x] = f(x) = \rho_R(x) \cap \rho_S(x) = R[x] \cap S[x]$ for all $x \in \beta \NN$. Thus, $T = R \cap S$. It follows that $R \cap S \in \CR(\beta \NN)$, which implies that $R \cap S = R \wedge S$, contradicting our choice of $R$ and $S$. Therefore, $\cap \colon \V(\beta \NN) \times \V(\beta \NN) \to \V(\beta \NN)$ is not continuous.
\end{proof}

Having reviewed the Vietoris construction, we turn to the proof that $\CCR(\beta M)$ is a locally Stone frame. 
Let $X$ and $Y$ be compact Hausdorff spaces. We recall that $\Omega(X)$ denotes the frame of open subsets of $X$, that $\Gamma(X)$ denotes the coframe of closed subsets of $X$, and that $\C(X, Y)$ denotes the set of continuous functions from $X$ to $Y$. We also write $\CR(X, Y)$ for the set of continuous relations from $X$ to $Y$. The set $\CR(X, Y)$ is of course ordered by inclusion, but if $Y$ is a poset, we order $\C(X,Y)$ pointwise: $f \leq g$ if $f(x) \leq g(x)$ for all $x \in X$.

Let $M$ be a set. Define $\kappa \: \CR(\beta M) \to \C(\beta M, \V(\beta M))$ by $\kappa(R) = \rho_R$. Then $\kappa$ is a bijection by the universal mapping property of $\mathcal V$. Also, define $\lambda\: \C(\beta M, \V(\beta M)) \to \C(M, \V(\beta M))$ by $\lambda(f) = f|_M$. Then $\lambda$ is a bijection by the universal mapping property of $\beta$. Finally, define $\mu\: \C(M, \V(\beta M)) \to \Gamma(M \times \beta M)$ by $\mu(g) = \bigcup_m \{m\} \times g(m)$. Then $\mu$ is a well-defined bijection because functions from $M$ to $\V(\beta M)$ correspond to closed subsets of $M \times \beta M$, which is homeomorphic to a disjoint union of copies of $\beta M$. We prove that these bijections are order isomorphisms.

\begin{theorem} \label{thm: CR(beta M)}
For any set $M$, the bijections
$$
\begin{tikzcd}
\CR(\beta M)
\arrow{r}{\kappa}
&
\C(\beta M, \V(\beta M))
\arrow{r}{\lambda}
&
\C(M, \V(\beta M))
\arrow{r}{\mu}
&
\Gamma(M \times \beta M)
\end{tikzcd}
$$
are order isomorphisms.
\end{theorem}

\begin{proof}
First, we reason that for all $R_1, R_2 \in \CR(\beta M)$,
\begin{align*}
\kappa(R_1) \leq \kappa(R_2) &\EV \rho_{R_1} \leq \rho_{R_2} \EV (\forall x \in \beta M)\; \rho_{R_1}(x) \subseteq \rho_{R_2}(x) \\ & \EV (\forall x \in \beta M)\; R_1[x] \subseteq R_2[x] \EV R_1 \subseteq R_2.
\end{align*}
Thus, $\kappa$ is an order isomorphism.

Second, we reason that for all $f_1, f_2 \in \C(\beta M, \V(\beta M))$,
\begin{align*}
\lambda(f_1) \leq \lambda(f_2) &\EV (\forall m \in M)\; f_1(m) \subseteq f_2(m) \\ & \EV (\forall x \in \beta M)\; f_1(x) \subseteq f_2(x) \EV f_1 \leq f_2.
\end{align*}
For the forward direction of the second equivalence, we appeal to the facts that $M$ is dense in $\beta M$ and that the inclusion relation on $\V(\beta M)$ is closed. Indeed, if $\{m_i\}$ is a net that converges to $x$ and $f_1(m_i) \subseteq f_2(m_i)$, then $f_1(x) \subseteq f_2(x)$. Thus, $\lambda$ is an order isomorphism.

Third, we reason that for all $g_1, g_2 \in \C(M, \V(\beta M))$,
\begin{align*}
\mu(g_1) \subseteq \mu(g_2) & \EV \bigcup_{m \in M} \{m\} \times g_1(m) \subseteq \bigcup_{m \in M} \{m\} \times g_2(m) \\ & \EV (\forall m \in M)\; g_1(m) \subseteq g_2(m) \EV g_1 \leq g_2.
\end{align*}
Thus, $\mu$ is an order isomorphism.
\end{proof}

We have shown that $\CR(\beta M)$ is isomorphic to $\Gamma (M \times \beta M)$, and thus, $\CCR(\beta M)$ is isomorphic to $\Omega(M\times\beta M)$. Since $M\times\beta M$ is a locally Stone space, it follows by \cref{thm: locally Stone} that $\CCR(\beta M)$ is a locally Stone frame.

\begin{corollary}
For each complete atomic boolean algebra $B$, the frame $\NO(B)$ is locally Stone. However, $\NO(\wp(\NN))$ is not a coframe, not completely distributive, and not a boolean algebra.
\end{corollary}

\begin{proof}
It is well known that the Stone space of $B$ is homeomorphic to $\beta M$ for some set $M$. Thus, $\NO(B)$ is isomorphic to
$\Omega(M \times \beta M)$ and hence is a locally Stone frame. 
For the second claim,  
it is sufficient to observe that  
$\Omega(\mathbb N \times \beta \mathbb N)$ is not a coframe.
\end{proof}

\section{The structure of $\TNO(B)$} \label{sec: tense operators}

In this final section, we study the poset $\TNO(B)$ of tense necessity operators on a boolean algebra $B$. These correspond to interior relations on the Stone space of $B$. As with $\NO(B)$, in general, we only have that $\TNO(B)$ is a bounded meet-semilattice that is not necessarily distributive. But if $B$ is complete, then $\TNO(B)$ is a frame. It is spatial iff $B$ is atomic, and in that case, $\TNO(B)$ is itself an atomic boolean algebra. Dual results hold for the poset $\TPO(B)$ of tense possibility operators.

The next definition goes back to J\'onsson and Tarski \cite{JT1951}. 

\begin{definition}\label{def: tense}
Let $B$ be a boolean algebra.
Two possibility operators $\Diamond_F$ and $\Diamond_P$ on $B$ are {\em conjugate} if 
\[
\Diamond_F a \meet b = 0 \ \ \Longleftrightarrow \ \ \Diamond_P b \meet a =0
\]
for all $a,b\in B$. The operator $\Diamond_F$ is a \emph{tense possibility operator} if it has a conjugate in this sense. In this case, $\Box_F$ is a \emph{tense necessity operator}. Let $\TNO(B)$ be the set of these tense necessity operators, and let $\TPO(B)$ be the set of these tense possibility operators.
\end{definition}

The above definition is 
equivalent to a number of definitions in the literature (see, e.g., \cites{Gol1992,GV1999,Ven2007}). For the reader's convenience, we collect them in the next proposition and add a short proof.

\begin{proposition}\label{prop: tense necessity operator}
Let $B$ be a boolean algebra, and let $\Box_F \in \NO(B)$. The following are equivalent:
\begin{enumerate}
\item $\Box_F$ is a tense necessity operator.
\item There exists a possibility operator $\Diamond_P$ on $B$ such that $\Diamond_P a \leq b$ iff $a \leq \Box_F b$ for all $a, b \in B$.
\item There exists a possibility operator $\Diamond_P$ on $B$ such that $a \le \Box_F \Diamond_P a$ and $\Diamond_P\Box_F a\le a$ for all $a\in B$.
\item There exists a possibility operator $\Diamond_P$ on $B$ such that $a \le \Box_F \Diamond_P a$ and $a\le \Box_P \Diamond_F a$ for all $a\in B$.
\item The meet $\bigwedge\{x \in B \mid a \le \Box_F x\}$ exists for each $a\in B$, and $\Box_F$ preserves these meets.
\end{enumerate}
\end{proposition}

\begin{proof}
(1)$\Leftrightarrow$(2): $\Diamond_F$ and $\Diamond_P$ are conjugate iff $\Diamond_F a \wedge b = 0 \Leftrightarrow \Diamond_P b \wedge a = 0$ for all $a,b\in B$. This is equivalent to $\Diamond_F a \le \neg b$ iff $\Diamond_P b \le \neg a$ for all $a,b\in B$. This, in turn, is equivalent to $b \le \Box_F \neg a$ iff $\Diamond_P b \le \neg a$ for all $a,b\in B$, which is equivalent to $b \le \Box_F c$ iff $\Diamond_P b \le c$ for all $b,c\in B$.

(2)$\Leftrightarrow$(3): Condition (2) states that $\Box_F$ has a left adjoint. It is well known that this is equivalent to Condition (3); see, e.g., \cite[p.~16]{Wood2004}.

(3)$\Leftrightarrow$(4): We have that $\Diamond_P\Box_F a\le a$ iff $\neg a \le \neg\Diamond_P\Box_F a$ iff $\neg a \le \Box_P\Diamond_F \neg a$. Since $a\in B$ is arbitrary, $\Diamond_P\Box_F a\le a$ for all $a\in B$ iff $c \le \Box_P\Diamond_F c$ for all $c\in B$. The equivalence follows.

(2)$\Leftrightarrow$(5): It is well known that a necessity operator $\Box_F$ has a left adjoint iff for each $a\in B$ the meet $\bigwedge\{x \in B \mid a \le \Box_F x\}$ exists and $\Box_F$ preserves these meets; see, e.g., \cite[p.~22]{Wood2004}.
\end{proof}

\begin{remark}
We view $\TNO(B)$ as a subposet of $\NO(B)$ and $\TPO(B)$ as a subposet of $\PO(B)$. By \cref{def: tense}, the duality $\Box \mapsto \neg \Box \neg$ between $\NO(B)$ and $\PO(B)$ restricts to a duality between $\TNO(B)$ and $\TPO(B)$.
\end{remark}

\begin{proposition} \label{prop: TNO(B)}
For each boolean algebra $B$, the poset $\TNO(B)$ is a bounded meet-semilattice.
\end{proposition}

\begin{proof}
    It is enough to observe that the meet of two tense necessity operators in $\NO(B)$ is a tense necessity operator. This follows from \cref{prop: tense necessity operator}(2). 
\end{proof}

We recall \cite[Def.~5.1]{BGHJ2019} that a continuous relation $R$ on a Stone space $X$ is an {\em interior relation} if $R[U]$ is clopen for each clopen $U\subseteq X$. Note that if $R$ is symmetric, then $R$ is interior iff $R$ is continuous.
Let $\IR(X)$ be the poset of interior relations on $X$, ordered by inclusion. We have the following analog of \cref{lem: JT}. 

\begin{lemma}\label{prop: tno when complete}
    Let $B$ be a boolean algebra, and let $X$ be its Stone space.
    The dual isomorphism of \cref{lem: JT} restricts to a dual isomorphism between the posets $\TNO(B)$ and $\IR(X)$.
\end{lemma}

\begin{proof}
To simplify the notation, we identify $B$ with ${\rm Clop}(X)$. Let $R$ be a continuous relation on $X$, and let $\Box_F\: V \mapsto X\setminus R^{-1}[X\setminus V]$. It is sufficient to show that $R$ is interior iff $\Box_F$ is a tense necessity operator on ${\rm Clop}(X)$. If $R$ is interior, then $\Diamond_P\: U \mapsto R[U]$ is a left adjoint to $\Box_F$, and hence $\Box_F$ is a tense necessity operator on ${\rm Clop}(X)$ by \cref{prop: tense necessity operator}.

Conversely, assume that $\Box_F$ is a tense necessity operator on ${\rm Clop}(X)$. By \cref{prop: tense necessity operator}, $\Box_F$ has a left adjoint $\Diamond_P \: U \mapsto  \bigwedge \{ W \in {\rm Clop}(X) \mid U \subseteq \Box_F W \}$. We show that $R[U] = \Diamond_P U$ for each $U \in {\rm Clop}(X)$. The inclusion $R[U] \subseteq \Diamond_P U$ follows from $U \subseteq \Box_F \Diamond_P U$ because $U \subseteq \Box_F W$ iff $R[U] \subseteq W$ for all $W \in {\rm Clop}(X)$. For the other inclusion, we calculate that
\begin{eqnarray*}
\Diamond_P U &=& \bigwedge \{ W \in {\rm Clop}(X) \mid  R[U] \subseteq W \} \\
&=& {\sf int} \bigcap \{ W \in {\rm Clop}(X) \mid  R[U] \subseteq W \} \\
&=& {\sf int} R[U] \subseteq R[U].
\end{eqnarray*}
Thus, $R[U] = \Diamond_P U \in {\rm Clop}(X)$ for all $U \in {\rm Clop}(X)$, and therefore, $R$ is interior. We conclude that $R$ is interior iff $\Box_F$ is a tense necessity operator on ${\rm Clop}(X)$. The lemma follows.
\end{proof}

The above lemma together with \cref{prop: TNO(B)} yields that $\IR(X)$ is a bounded join-semilattice for each Stone space $X$. We can also see this directly by observing that the union of two interior relations is an interior relation. 

\begin{proposition}
Let $X$ be a Stone space. Then $\IR(X)$ is a bounded join-semilattice that in general is neither distributive nor a lattice. 
\end{proposition}

\begin{proof}
This proof follows the pattern of the proof of \cref{thm: counterexamples}. The relations $D$, $R$, and $S$ of \cref{def: DRS} are interior relations on $X = \alpha \NN$ by \cref{lem: DRS interior}. As before, we have that $D \cap R = \{(2n, 2n) \mid n \in \NN\}$ and $D \cap S = \{(2n +1, 2n +1) \mid n \in \NN\}$. By \cref{lem: CR under D}, the interior relations under $D \cap R$ are all finite, and it is likewise for $D \cap S$. Thus, there is no greatest interior relation under $D \cap R$ nor under $D \cap S$; the 
meets $D \wedge R$ and $D \wedge S$ do not exist in $\IR(\alpha \NN)$. Therefore, $\IR(\alpha \NN)$ is not a lattice.
Furthermore, no interior relations $R' \subseteq R$ and $S' \subseteq S$ can satisfy $R' \union S' = D$, despite the inclusion $D \subseteq R \cup S$. Therefore, the join-semilattice $\IR(\alpha \NN)$ is not distributive.
\end{proof}

We thus obtain that $\TNO(B)$ is in general neither distributive nor a lattice. When $B$ is complete, $\TNO(B)$ is better behaved.

\begin{proposition} \label{prop: Weaver}
Let $B$ be a complete boolean algebra.
\begin{enumerate}
\item A function $\Box_F:B\to B$ is a tense necessity operator iff for all $\mathcal{S} \subseteq B$, $$\Box_F \Meet \mathcal{S} = \Meet \{\Box_F a \mid a \in \mathcal{S}\}.$$
\item {\em (cf.~\cite{Weav2011}*{Def.~1.2})} There is an order isomorphism between $\TNO(B)$ and the set $\mathcal{TNR}(B)$ of relations $R$ on $B$ such that for all $\mathcal{S}, \T \subseteq B$,
\begin{equation}\label{eq: Weaver condition}\tag{$\dagger$}
\left(\Meet\mathcal{S}, \Meet\T\right) \in R \ \ \Longleftrightarrow \ \ (a, b) \in R \ \ \text{for all $a \in \mathcal{S}$ and $b \in \T$}.
\end{equation}
We order these relations by inclusion. The order isomorphism is $$\Box_F \mapsto \{ (a, b) \in B \times B \mid \Box_F a \join b = 1\}.$$
\end{enumerate}
\end{proposition}

\begin{proof}
(1) This follows from \cref{prop: tense necessity operator} (also see \cite{Wood2004}*{p.~22}).

(2) Let $\Box_F$ be a tense necessity operator on $B$, and let $R = \{ (a, b) \in B \times B \mid \Box_F a \join b = 1\}$. For all $\mathcal{S}, \T \subseteq B$, we have:
\begin{align*}
\left(\Meet \mathcal{S}, \Meet \T\right) \in R
& \EV
\left(\Box_F \Meet \mathcal{S}\right)  \join \Meet \T =  1 
\\ & \EV
\Meet \{\Box_F a \mid a \in \mathcal{S}\} \join \Meet \T = 1 
\\ & \EV
\Meet \{\Box_F a \join b \mid a \in \mathcal{S},\, b \in \T\} = 1
\\ & \EV 
\Box_F a \join b = 1 \ \ \text{for all $a \in \mathcal{S}$ and $b \in \T$}
\\ & \EV
(a, b) \in R \ \ \text{for all $a \in \mathcal{S}$ and $b \in \T$}.
\end{align*}
Thus, $i\: \Box_F \mapsto \{ (a, b) \in B \times B \mid \Box_F a \join b = 1\}$ is a well-defined function from $\TNO(B)$ to $\TNR(B)$. It is monotone because $\Box_F \leq \Box_F'$ and $\Box_F a \join b = 1$ 
imply that $\Box_F' a \join b = 1$.

Let $R \in \TNR(B)$, and let $\Box_F a = \neg \Meet R[a]$ for all $a \in B$. Let $\mathcal{S} \subseteq B$. Then, the equality $\Box_F \Meet \mathcal{S} = \Meet \{\Box_F a \mid a \in \mathcal{S}\}$ is equivalent to the equality $\Meet R[\Meet\mathcal{S}] = \bigvee \{ \Meet R[a] \mid a \in \mathcal{S}\}$. We prove the latter equality by proving the two inequalities. Before doing so, we note that condition (\ref{eq: Weaver condition}) may be rephrased as the pair of requirements that $b_0 \in R[\Meet \mathcal{S}_0]$ iff $b_0 \in R[a]$ for all $a \in \mathcal{S}_0$ and that $\Meet \T_0 \in R[a_0]$ iff $b \in R[a_0]$ for all $b \in \T_0$, where $a_0, b_0 \in B$ and $\mathcal{S}_0, \T_0 \subseteq B$. Taking $\T_0 = R[a_0]$, we find that $\Meet R[a_0] \in R[a_0]$. Furthermore, taking $\T_0 = R[a_0] \cup \{b_0\}$, we find that $b_0 \geq \Meet R[a_0]$ implies $b_0 \in R[a_0]$. Thus, $R[a_0] = \{b \in B \mid b \geq \Meet R[a_0]\}$. Finally, taking $\mathcal{S}_0 = \{a_1, a_2\}$, we find that $a_1 \leq a_2$ implies $R[a_1] \subseteq R[a_2]$ for all $a_1, a_2 \in B$.

For the first inequality, we reason that $\Meet \mathcal{S} \leq a$ for all $a \in \mathcal{S}$, so $R[\Meet \mathcal{S}] \subseteq R[a]$ for all $a \in \mathcal{S}$, and thus, $\Meet R[a] \leq \Meet R[\Meet \mathcal{S}]$ for all $a \in \mathcal{S}$. Therefore, $\bigvee\{ \Meet R[a] \mid a \in \mathcal{S}\} \leq \Meet R[\Meet \mathcal{S}]$. For the second inequality, we reason that $\bigvee\{ \Meet R[a] \mid a \in \mathcal{S}\} \geq \Meet R[a]$ for all $a \in \mathcal{S}$, and thus, $\bigvee\{ \Meet R[a] \mid a \in \mathcal{S}\} \in R[a]$ for all $a \in \mathcal{S}$. We infer that $\bigvee\{ \Meet R[a] \mid a \in \mathcal{S}\} \in R[\Meet \mathcal{S}]$. Therefore, $\bigvee\{ \Meet R[a] \mid a \in \mathcal{S}\} \geq \Meet R[\Meet \mathcal{S}]$. We conclude that $\bigvee\{ \Meet R[a] \mid a \in \mathcal{S}\} = \Meet R[\Meet \mathcal{S}]$ and hence that $\Meet \{\Box_F a \mid a \in \mathcal{S}\} = \Box_F \Meet \mathcal{S}$. Thus, $j\: R \mapsto \Box_F$, where $\Box_F a = \neg \Meet R[a]$, is a well-defined function from $\TNR(B)$ to $\TNO(B)$.
It is monotone because $R \subseteq R'$ implies that $\neg \Meet R[a] \leq \neg \Meet R'[a]$ for all $a \in B$. 

We now have two monotone functions, $i$ and $j$, between $\TNO(B)$ and $\TNR(B)$, and it remains to show that they are inverses. For each $\Box_F \in \TNO(B)$ and all $a \in B$, we compute that
$
j(i(\Box_F))(a) = \Not \Meet i(\Box_F)[a] = \Not \Meet \{b \in B \mid \Box_F a \join b = 1\} = \Not \Not \Box_F a = \Box_F a.
$
Thus, $j(i(\Box_F)) = \Box_F$ for all $\Box_F \in \TNO(B)$. Also, for each $R \in \TNR(B)$, we compute that
\begin{align*}
i(j(R)) & = \{(a,b) \in B \times B \mid j(R) a \join b =1\}
= \{(a,b) \in B \times B \mid \Not \Meet R[a] \join b =1\}
\\ & =
\{(a,b) \in B \times B \mid \Meet R[a] \leq b \}
=
\{(a,b) \in B \times B \mid b \in R[a]\}
= R.
\end{align*}
Thus, $i$ and $j$ are inverses, and so $i$ is an order isomorphism.
\end{proof}

The set $\TNR(B)$ was essentially first defined by Weaver \cite{Weav2011} in the special case that $B$ is a measure algebra. Recall that the boolean algebra $B$ of measurable sets modulo null sets is complete for any strictly localizable measure space, e.g., for a probability space \cites{Frem2003,Seg1951}. A \emph{measurable relation} on such a measure space is then a binary relation on $B \setminus\{0\}$ such that $\bigvee \mathcal{S}$ is related to $\bigvee \T$ iff some $a \in \mathcal{S}$ is related to some $b \in \T$ \cite{Weav2011}*{Def.~1.2}. Implicitly, $\mathcal{S}$ and $\T$ are nonempty subsets of $B\setminus\{0\}$. This condition extends naturally to arbitrary subsets of $B$ with the effect that nothing is related to $0$ and that $0$ is related to nothing. The elements of $\TNR(B) \iso \TNO(B)$ are then duals of these measurable relations.

We now show that $\TNO(B)$ is a frame for any complete boolean algebra $B$. In this sense, the measurable relations on a strictly localizable measure space obey intuitionistic logic.

\begin{proposition}
For each extremally disconnected Stone space $X$, the poset $\IR(X)$ is a coframe. Therefore, for each complete Boolean algebra $B$, the poset $\TNO(B)$ is a frame. 
\end{proposition}

\begin{proof}
The proof of \cref{thm: coframe} applies in its entirety with $\IR(X)$ in place of $\CR(X)$. \cref{lem: CR under D} implies that the relation $D_X \cap (U \times U)$ is interior for each clopen $U \subseteq X$ because it is symmetric.
\end{proof}

Each set $M$ becomes a strictly localizable measure space when it is equipped with counting measure. In this case, the measure algebra is just the powerset algebra $\wp(M)$, and the measurable relations are just the relations on $M$ \cite{Weav2011}*{Prop.~1.3}. Thus, by \cref{prop: tno when complete}, $\TNO(\wp(M)) \iso \wp(M \times M)$. This isomorphism is also a consequence of the well-known Thomason duality in modal logic \cite{Tho1975}. We provide a third proof, which utilizes Esakia's Lemma (see, e.g., \cite[Lem.~10.27]{CZ1997}).

\begin{lemma} [Esakia's Lemma]
    Let $R$ be a binary relation on a Stone space $X$.
    \begin{enumerate}
        \item Suppose $R[x]$ is closed for each $x\in X$. If $\{ U_\alpha \}$ is a down-directed family of clopens, then $\bigcap_\alpha R^{-1}[U_\alpha] = R^{-1}[\bigcap_\alpha U_\alpha]$.
        \item Suppose $R^{-1}[x]$ is closed for each $x\in X$. If $\{ U_\alpha \}$ is a down-directed family of clopens, then $\bigcap_\alpha R[U_\alpha] = R[\bigcap_\alpha U_\alpha]$.
    \end{enumerate}
\end{lemma}

\begin{proof}
    For (1) we refer to \cite[Lem.~3.3.12]{Esakia2019} and \cite[Lem.~2.17]{BBH15}, and (2) is proved similarly.
\end{proof}
   
\begin{theorem} \label{thm: rep for IR(M)}
Let $M$ be a set. There is an order isomorphism $f\: \IR(\beta M) \to \wp(M \times M)$ that is defined by $f(R) = R \cap (M \times M)$.
Its inverse $g\: \wp(M \times M) \to \IR(\beta M)$ is defined by $g(Q) = \cl\, Q$.
\end{theorem}

\begin{proof}
It is clear that $f[\IR(\beta M)] \subseteq \wp(M \times M)$. To show that $g[\wp(M \times M)] \subseteq \IR(\beta M)$, let $Q \subseteq M \times M$
and $P$ be its closure in $\beta M \times \beta M$. Clearly $P$ is a closed relation. To see that it is interior, it is enough to show that $P^{-1}[V]$ and $P[V]$ are clopen for each clopen $V \subseteq \beta M$. Since the two proofs are similar, we only prove that $P^{-1}[V]$ is clopen. 
By \cref{lem: auxiliary}, 
\[
P^{-1}[V] = \cl\, Q^{-1}[V] =
\cl\, Q^{-1}[V \cap M],
\]
and thus, $P^{-1}[V]$ is clopen for each clopen $V \subseteq \beta M$. Symmetrically, $P[V]=\cl\,Q[V\cap M]$, and thus, $P[V]$ is clopen for each clopen $V \subseteq \beta M$. We conclude that $P \in \IR(\beta M)$. 

Thus, $f$ and $g$ are indeed functions $\IR(\beta M) \to \wp(M \times M)$ and $\wp(M \times M) \to \IR(\beta M)$, respectively. Both are clearly monotone. Moreover, for all $Q \in \wp(M \times M)$,
\[
f(g(Q)) = \cl Q \cap (M\times M) = Q.
\]
It is left to show that $g(f(R))=R$ for all $R \in \IR(\beta M)$.
In other words, it is left to show that $R = \cl(R \cap (M\times M))$ for all $R \in \IR(\beta M)$. 

Let $R \in \IR(\beta M)$. Let 
$Q = R \cap (M \times M)$ and $P = \cl Q$. Clearly $P \subseteq R$. To prove the other inclusion, we show that $R[x] \subseteq P[x]$ for all $x \in \beta M$.
For each $x \in M$, since $\{x\}$ is clopen and $R$ is interior, $R[x]$ is clopen, and thus
$$
R[x] = \cl(R[x] \cap M) = \cl(Q[x]) = P[x].
$$
Since $R$ is continuous, we have $R[\cl A] = \cl R[A]$ for each $A\subseteq \beta M$ \cite[Thm.~3.1.2]{Esakia2019}. Hence, for each clopen $U \subseteq \beta M$,
\begin{align*}
R[U] &= R[ \cl(U\cap M) ] = \cl R[U\cap M] = \cl \left(\bigcup\{ R[x] \mid x \in U \cap M\}\right) \\ & = \cl \left(\bigcup\{ P[x] \mid x \in U \cap M\}\right) = \cl P[U \cap M]  \subseteq P[U].
\end{align*}
Applying Esakia's Lemma to both $R$ and $P$,
we conclude that 
\begin{eqnarray*}
    R[x] &=& R\left[ \bigcap \{ U \in {\rm Clop}(X) \mid x \in U \}\right] = \bigcap \{ R[U] \mid x \in U \} \\
    & \subseteq & \bigcap \{ P[U] \mid x \in U \} = P\left[ \bigcap \{ U \in {\rm Clop}(X) \mid x \in U \} \right] \\
    &=& P[x]
\end{eqnarray*}
for all $x \in \beta M$. Thus, $R \subseteq P$, and therefore, $R = P = \cl\, Q = \cl(R \cap (M \times M)) = g(f(R))$.
We have shown that $f$ and $g$ are inverses, and hence $f\: \IR(\beta M) \to \wp(M \times M)$ is an order isomorphism.
\end{proof}

\begin{remark}
The composition of the order isomorphism of \cref{thm: rep for IR(M)} and
the dual isomorphism of \cref{prop: tno when complete} is given by
$$
\Box_F \ \mapsto \ \{(x,y) \in M \times M \mid y \in \Box_F ( M \setminus \{x\})\}.
$$
\end{remark}

Combining \cref{prop: tno when complete,thm: rep for IR(M)}, we obtain the following.

\begin{corollary}\label{cor: iso for TNO}
Let $B$ be a complete boolean algebra. The following are equivalent:
\begin{enumerate}
    \item $B$ is atomic.
    \item $\TNO(B)$ is a complete atomic boolean algebra.
    \item $\TNO(B)$ is a spatial frame.
\end{enumerate}
\end{corollary}

\begin{proof}
(1)$\Rightarrow$(2): If $B$ is atomic, then $B$ is isomorphic to $\wp(M)$ for some set $M$. Therefore, the Stone space of $B$ is homeomorphic to $\beta M$, so  \cref{prop: tno when complete} implies that $\TNO(B)$ is dually isomorphic to  $\IR(\beta M)$. But the latter is isomorphic to $\wp(M \times M)$ by \cref{thm: rep for IR(M)}. Since $\wp(M \times M)$ is dually isomorphic to itself, we conclude that $\TNO(B)$ is complete and atomic.

(2)$\Rightarrow$(3): Every complete atomic boolean algebra is a spatial frame.

(3)$\Rightarrow$(1): 
Let $X$ be the Stone space of $B$. Then $X$ is extremally disconnected, and $\TNO(B)$ is dually isomorphic to $\IR(X)$ by \cref{prop: tno when complete}. Since $\TNO(B)$ is spatial, $\IR(X)$ is join-generated by join-primes. Joins in $\IR(X)$ coincide with those in $\CR(X)$, and hence we may apply \cref{lem: join-prime} to conclude that each join-prime of $\IR(X)$ is a singleton $\{(x,y)\}$ such that $x$ is isolated. (Although it is not needed for the proof, $y$ is also isolated because $R$ is interior.) Thus, for each clopen $U\subseteq X$, there exists a pair $(x,y)\in U\times U$ such that $x$ is isolated. Therefore, isolated points are dense in $X$, and hence $B$ is atomic.
\end{proof}

\begin{remark}
    Consider $M$ as a discrete topological space. Then $\TNO(\wp(M))$ is dually isomorphic to $\Gamma(M \times M)$ by the proof of \cref{cor: iso for TNO}, and $\NO(\wp(M))$ is dually isomorphic to $\Gamma(M \times \beta M)$ by \cref{thm: CR(beta M)}. It is the frame of subordination relations on $\wp(M)$ that is dually isomorphic to $\Gamma(\beta M \times \beta M)$ (see \cite[Sec.~2]{BBSV2017}).
\end{remark}

Our results for $\TNO(B)$ throughout \cref{sec: tense operators} have obvious dual analogues for $\TPO(B)$. A category-theoretic characterization of complete atomic boolean algebras and tense possibility operators was given in \cite{Kor2025}.

We conclude by outlining some possible directions for future work.
\begin{enumerate}
    \item As we pointed out in the introduction, $\NO(B)$ can be thought of as the set of endomorphisms in the category $\cat{BA^H}$ of boolean algebras and hemimorphisms between them. More generally, it is natural to study the set $\mathrm{Hom}_{\cat{BA^H}}(B_1,B_2)$. By an obvious generalization of \cref{prop: semilattice}, $\mathrm{Hom}_{\cat{BA^H}}(B_1,B_2)$ is a bounded meet-semilattice, and it is natural to study further properties of $\mathrm{Hom}_{\cat{BA^H}}(B_1,B_2)$.
    \item It is also natural to generalize our considerations to the category $\cat{MSL}$ of meet-semilattices (with $1$) and meet-semilattice homomorphisms (preserving $1$). The necessity operators on $S\in\cat{MSL}$ are simply the endomorphisms of $S$ in $\cat{MSL}$. Observe that \cref{prop: semilattice} holds for an arbitrary meet-semilattice, and \cref{prop: NO of CBA} for an arbitrary complete meet-semilattice. It remains open whether $\NO(S)$ is distributive or a (zero-dimensional) frame provided that $S$ is complete. 
    \item It would also be interesting to study tense operators in a more general setting. Since it is more natural to have conjugate operators on a lattice rather than a semilattice, the first step would be to study the structure of $\TNO(L)$ for an arbitrary (bounded) lattice $L$.  
\end{enumerate}

\bibliographystyle{amsplain}
\bibliography{references}

\end{document}